\documentclass[12pt]{article}

\usepackage{graphicx}
\usepackage{geometry}
\usepackage{amsfonts}
\usepackage{amsmath}
\usepackage{enumerate}
\usepackage{amsthm}
\usepackage{hyperref}
\usepackage{xcolor}

\begin{document}

\setcounter{page}{1} \setcounter{section}{0}
\newtheorem{theorem}{Theorem}[section]
\newtheorem{lemma}[theorem]{Lemma}
\newtheorem{corollary}[theorem]{Corollary}
\newtheorem{proposition}[theorem]{Proposition}
\newtheorem{observation}[theorem]{Observation}
\newtheorem{definition}[theorem]{Definition}
\newtheorem{claim}{Claim}
\newtheorem{conjecture}[theorem]{Conjecture}
\newtheorem{problem}[theorem]{Problem}

\title{On the Ban-Linial Conjecture}
\author{
{\sc Matt DeVos}$\,^{a}$, {\sc Kathryn Nurse}$\,^{b}$\\
\mbox{}\\
{\small $(a)$ Simon Fraser University, Burnaby Canada}\\
{\small $(b)$ \'Ecole Normale Sup\'erieure, Paris France}
}

\date{\today}

\maketitle

\begin{abstract}
Let $G$ be a graph and let $\{X_0,X_1\}$ be a partition of $V(G)$.  This partition is called external or unfriendly if every $x \in X_i$ has at least as many neighbours in $X_{1-i}$ as in $X_i$.  Every maximum edge-cut gives rise to an external partition, so these partitions are always guaranteed to exist.  However, it remains a challenge to find such partitions with additional restrictions.  Ban and Linial have conjectured that in the case when $G$ is cubic, there always exists an external partition $\{X_0,X_1\}$ for which $-2 \le |X_0| - |X_1| \le 2$.  We prove this in two special cases: whenever $G$ can be decomposed into a cycle and a tree, and whenever $G$ has a cubic tree $T$ for which $G - E(T)$ is bipartite.
\end{abstract}

\section{Introduction}

Throughout we assume graphs are finite and simple, and we use \cite{BondyJ.A.JohnAdrian2008Gt/J} for terms not defined here. Given a set $S$, a \emph{split} of $S$ is a pair of sets $(T,U)$ whose disjoint union is equal to $S$ (i.e., $T \sqcup U = S$). For a graph $G = (V,E)$ we define a \emph{split} of $G$ to be a split of $V$.  We let $e(G) = |E|$ and for a split $(X,Y)$ of $G$ we let $E(X,Y) = 
\{ xy \in E(G) \mid \mbox{$x \in X$ and $y \in Y$} \}$ and $e(X,Y) = |E(X,Y)|$.  We say that the split $(X,Y)$ is \emph{external} \emph{(internal)} if the graph $H = G - E(X,Y)$ satisfies $\mathrm{deg}_H(v) \le \frac{1}{2} \mathrm{deg}_G(v)$ $(\mathrm{deg}_H(v) \ge \frac{1}{2} \mathrm{deg}_G(v))$ for every $v \in V(G)$.  

For every graph $G$, if we choose $(X,Y)$ to be a split for which $e(X,Y)$ is maximum (i.e., a maxcut of $G$), then it is immediate that $(X,Y)$ is external.  So it is not difficult to find one external split.  However, finding external splits with additional conditions remains quite challenging.  The following conjecture in this vein is the motivation for our work.

\begin{conjecture}[Ban-Linial]
Every cubic graph has an external split $(X,Y)$ for which $-2 \le |X| - |Y| \le 2$.
\end{conjecture}

As motivation, notice that this conjecture is satisfied by every 3-edge-colourable cubic graph $G$.  To see this, let $\{M_1, M_2, M_3\}$ be three pairwise disjoint perfect matchings given by a 3-edge-colouring, and let $(X,Y)$ be a bipartition of $G - M_3$. Then $(X,Y)$ is an external split in $G$ with $|X| = |Y|$. The Petersen graph does not have an external split $(X,Y)$ with $|X| = |Y|$ and initially it seemed plausible this would be the only such connected cubic graph.  However, counterexamples to this stronger conjecture were found by Esperet, Mazzuoccolo and Tarsi \cite{EsperetL.2017FaBi, BanAmir2016IPoR}.  

The Ban-Linial conjecture has been verified for some classes of cubic graphs: Abreu et. al. verified this conjecture for bridgeless claw-free cubic graphs \cite{ABREU2018214}, and Zerafa has proved it for a family of snarks \cite{ZerafaJeanPaul2022BCat}.  However there are relatively few general results in this direction.  Our main contribution is Theorem \ref{main}, which resolves the Ban-Linial conjecture for cubic graphs having certain decompositions.  We say that a graph $G$ can be \emph{decomposed into a tree and a cycle} if there exists a tree $T$ and a cycle $C$ so that $\{E(T), E(C)\}$ is a partition of $E(G)$.  We call a tree $T$ \emph{cubic} if all vertices in $T$ have degree 1 or 3.  

\begin{theorem}
\label{main}
A cubic graph $G$ satisfies the Ban-Linial Conjecture if one of the following properties hold:
\begin{enumerate}
\item $G$ contains a cubic tree $T$ so that $G - E(T)$ is bipartite.
\item $G$ can be decomposed into a tree and a cycle.
\end{enumerate}
\end{theorem}

For a cubic graph $G$ an external split $(X,Y)$ has the property that every component of both $G[X]$ and $G[Y]$ has at most two vertices.  This suggests another notion that turns out to be quite rich.  We say that a split $(X,Y)$ of a graph $G$ is a $k$-\emph{split} if every component of both $G[X]$ and $G[Y]$ contains at most $k$ vertices. This is a special 2-colouring with clustering $k$, for more on cluster colouring see \cite{WoodSurv}. 

A split ($k$-split) $(X,Y)$ is called a \emph{bisection} \emph{($k$-bisection)} if $|X| = |Y|$.  So the Ban-Linial conjecture asserts that every cubic graph has a 1-split that is close to being a bisection.  In fact, $k$-bisections of cubic graphs are of special interest thanks to an approach toward Tutte's 5-flow conjecture developed by Jaeger \cite{6eaa05ae-624e-36ea-92fc-663066b4cf0a}.  We briefly sketch this here, but a more detailed description may be found in \cite{EsperetL.2017FaBi}.  Let $G$ be an oriented cubic graph with a nowhere-zero $k$-flow (i.e. a function $\phi : E(G) \rightarrow \{ \pm 1, \ldots, \pm (k-1) \}$ so that at every vertex $v$ the sum of $\phi$ on the edges directed toward $v$ is equal to the sum of $\phi$ on the edges directed away.)  For every edge $e$, if $\phi(e) < 0$, then we reverse $e$ and replace $\phi(e)$ by its negation (note that this preserves our nowhere-zero flow).  Since $\phi$ is a flow, it must satisfy the following cut-conservation property:
\begin{itemize}
\item For every split $(U,W)$ the sum of $\phi$ on the edges from $U$ to $W$ must equal the sum of $\phi$ on the edges from $W$ to $U$.  
\end{itemize}
It follows from this (and $\phi > 0$) that the orientation of our graph is strongly connected.  In particular, this gives us a split $(X,Y)$ of $G$ so that every vertex in $X$ $(Y)$ has outdegree (indegree) 1.  Note that by degree sums, we must have $|X| = |Y|$ so $(X,Y)$ is a bisection.  Let $H$ be a component of $G[X]$ and observe that $H$ has maximum outdegree 1, so $H$ is either a tree or is unicyclic.  Since our graph is strongly connected, the latter is not possible, so $H$ must be a tree.  Now there is exactly one edge directed from $V(H)$ to its complement (with flow value at most $k-1$) and there are $|V(H)| + 1$ edges directed to $V(H)$ from its complement (each with flow value at least 1).  Therefore, cut-conservation implies that $|V(H)| + 1 \le k-1$ so $|V(H)| \le k-2$ and $(X,Y)$ is a $(k-2)$-bisection of $G$.  Jaeger has shown how to reverse this process under an additional assumption.  More precisely, he proved that the presence of a nowhere-zero $k$-flow in a cubic graph $G$ is equivalent to the presence of a $(k-2)$-bisection that satisfies a certain condition for every edge cut.

\section{Proofs}

The Ban-Linial Conjecture concerns splits $(X,Y)$ that are close to external bisections, but have the defect that $|X|$ and $|Y|$ are close but might be not equal.  Our proofs naturally give an alternative form of deficiency here that we define next.  We say that a split $(X,Y)$ is \emph{nearly external} if there is at most one vertex $v$ so that $\mathrm{deg}_G(v) <  2 \mathrm{deg}_{G- E(X,Y)} (v)$.  Next we show that nearly external bisections imply the splits required for the Ban-Linial Conjecture.

\begin{lemma}
\label{nearlyexternal2banlinial}
If $G$ is a cubic graph with a nearly external bisection, then $G$ satisfies the Ban-Linial Conjecture.
\end{lemma}

\begin{proof}
Choose a nearly external bisection $(X,Y)$ so that $e(X,Y)$ is maximum and let $H = G - E(X,Y)$.  If $(X,Y)$ is external, we have nothing left to prove.  Otherwise, we may assume (without loss) that $y \in Y$ is the unique vertex with $d_H(y) \ge 2$.  If $(X \cup \{y\}, Y \setminus \{y\})$ is external the proof is complete.  Otherwise, there is a unique vertex $x \in X \cup \{y\}$ with two neighbours in this set. In this last case $X' = (X \setminus \{x\}) \cup \{y\}$ and $Y' = V \setminus X'$ contradict the choice of $(X,Y)$ and this completes the proof. 
\end{proof}

For a graph $G = (V,E)$ and a split $(X,Y)$ we define the \emph{discrepancy} of $(X,Y)$ to be
\[ \mathrm{disc}_G(X,Y) = e ( G[X] ) - e(G[Y]) \]
and we drop the subscript when the graph is clear from context.  For regular graphs, the discrepancy of a split is proportional to the difference in sizes of the sets as shown by the following observation.  Later we will take advantage of this to control $|X| - |Y|$ by discrepancy.

\begin{observation}
\label{disc2bisect}
If $(X,Y)$ is a split of a cubic graph, then 
\[ 3\big( |X| - |Y| \big) = 2 \, \mathrm{disc}(X,Y) \]
In particular, if $|\mathrm{disc}(X,Y)| \le 2$ then $(X,Y)$ is a bisection.
\end{observation}

\begin{proof}
To see that the equation holds, observe that both sides are equal to $\sum_{x \in X} d(x) - \sum_{y \in Y} d(y)$.  The second part follows immediately from the first.
\end{proof}

Next we prove the key lemma on cubic trees that we use to derive one part of our main theorem.  For a nonnegative integer $k$ and a graph $G$ we let $V_k(G) = \{ v \in V(G) \mid d(v) = k \}$.  If $T$ is a tree and $u,u' \in V_1(G)$ have a common neighbour, then we call $\{u,u'\}$ a \emph{pair of cherries}.  Trees without degree 2 vertices are guaranteed to have cherries whenever they have at least three vertices (consider the ends of a longest path) and we will take advantage of this fact repeatedly.

\begin{lemma}
\label{unrootedlemma}
Let $T$ be a cubic tree with $|V(T)| \ge 4$, let $(X,Y)$ be a split of $V_1(G)$, and let $\epsilon = \pm 1$.  There exists a split $(X^*,Y^*)$ of $T$ satisfying all of the following properties:
\begin{enumerate}
\item $X \subseteq X^*$ and $Y \subseteq Y^*$,
\item $\mathrm{disc}(X^*,Y^*) \in \epsilon \{0,1,2\}$
\item $T - E(X^*,Y^*)$ has at most one vertex of degree greater than one.
\end{enumerate}
\end{lemma}

\begin{proof}
We proceed by induction on $|V|$ with the base case $|V| =4$ holding by inspection.  Next suppose there exists a pair of cherries $\{w,w'\}$ both adjacent to the vertex $v \in V$ so that $w,w' \in X$.  In this case the result follows by applying induction to the tree $T - \{w,w'\}$ and the sets $X' = X \setminus \{w,w'\}$ and $Y' = Y \cup \{v\}$.  A similar argument holds if $w,w' \in Y$, so we may assume that we do not have a pair of cherries of the same colour.  If $|V| = 6$ or $|V| = 8$, then up to isomorphism and colour swapping there is just one possibility for $T$, and in both of these cases the result is straightforward to verify.  So we may now assume $|V| \ge 10$.  Next we will utilize some reductions that we describe with figures.  We will treat our splits $(X,Y)$ and $(X^*,Y^*)$ as (improper) colourings of the vertices, and we adopt the convention for our figures that a solid black vertex is in $X$ and an open white vertex is in $Y$.  

\bigskip

\noindent{\it Reduction 1.}

If $T$ contains a vertex $v$ with two neighbours both of which are adjacent with a pair of cherries, then we perform reduction 1 to reduce to a smaller tree $T'$ and flip the sign of $\epsilon$ as shown in the figure below.

\begin{center}
\includegraphics[height=2.5cm]{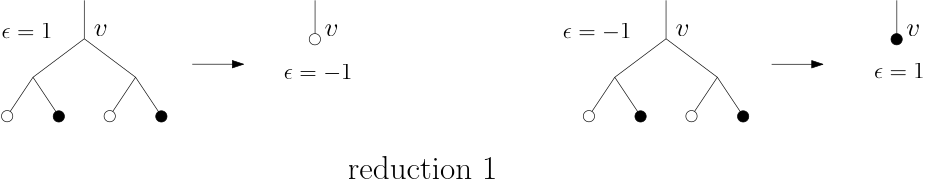}
\end{center}

Now we apply induction to colour the vertices of the smaller tree.  If we have the case on the left (right) in the above figure, we extend back to a colouring of $T$ by assigning both neighbours of $v$ in $V(T) \setminus V(T')$ the colour black (white).  This operation either increases or decreases the discrepancy by 2, but in either case we have the desired colouring of $T$.

\bigskip

\noindent{\it Reduction 2.}

If $T$ contains a vertex $v$ adjacent to a leaf vertex and another vertex $w$ so that $w$ is adjacent with a pair of cherries, then we perform reduction 2 to reduce $T$ to a smaller tree $T'$

\begin{center}
\includegraphics[height=2.5cm]{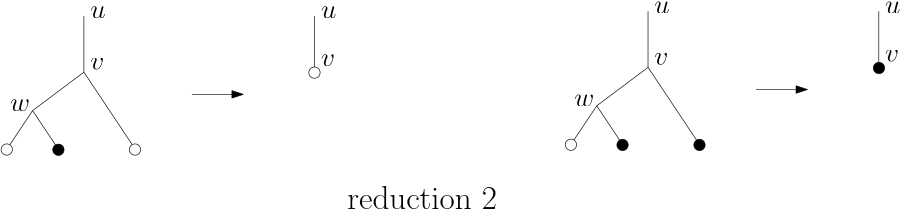}
\end{center}

Now we apply induction to colour the vertices of this smaller tree $T'$ using the original parameter $\epsilon$.  Let $u$ be the unique neighbour of $v$ that is not a leaf and is not equal to $w$ (as shown in the figure).  Since the two versions of reduction 2 are similar, we will assume that we have applied the one on the left (so $v$ is coloured white in $T'$).  If $u$ is coloured black in the resulting colouring of $T'$ then we keep this colour for $u$ and assign $w$ the colour black.  The resulting colouring of $T$ has the same discrepancy as that of $T'$ and will thus satisfy the lemma.  If $u$ is coloured white in the resulting colouring of $T'$ then we recolour $v$ to be black and assign $w$ the colour white.  Again the resulting colouring of $T$ has the same discrepancy as that of $T'$ and this satisfies the lemma.  

\bigskip

To complete the proof, it now suffices to show that one of our two reductions can be performed.  To see this, recall that every pair of cherries in $T$ have opposite colours, one black and one white.  As a bookkeeping device, for every pair of cherries with one white and the other black, let $v$ be the common neighbour, delete these cherries and assign $v$ the colour red.   After this operation has been performed wherever possible, the resulting tree is still cubic, and must have a pair of cherries at least one of which is coloured red.  In all cases, this pair of cherries gives us a reduction of one of our two types.
\end{proof}

With this lemma in place we are ready to prove the first part of our main theorem (restated here).

\begin{theorem}\label{treeBip}
A cubic graph $G$ satisfies the Ban-Linial conjecture if it has a cubic tree $T \subseteq G$ so that $G - E(T)$ is bipartite.
\end{theorem}

\begin{proof}
Choose a proper 2-colouring $(X_0,Y_0)$ of the bipartite graph $G - E(T)$.  Now apply the previous lemma to the tree $T$ and the pair $(X_0 \cap V(T), Y_0 \cap V(T))$ to obtain a split $(X^*,Y^*)$ of $T$.  Define $X = X_0 \cup X^*$ and $Y = Y_0 \cup Y^*$.  By the previous lemma we have $|\mathrm{disc}(X,Y)| = |\mathrm{disc}(X^*,Y^*)| \le 2$ and it follows from this and Observation \ref{disc2bisect} that $(X,Y)$ is a bisection of $G$.  Furthermore, $G - E(X,Y)$ is equal to the graph $T - E(X^*,Y^*)$ plus some isolated vertices.  It follows from this and the previous lemma that $(X,Y)$ is nearly external.  Finally, applying Lemma \ref{nearlyexternal2banlinial} shows that $G$ satisfies the Ban-Linial Conjecture.
\end{proof}

For the second part of the main theorem, we have a cubic graph $G$ that can be decomposed into a tree $T$ and a cycle $C$.  In the case when $C$ has even size, the result follows from the first part of the theorem.  In the remaining case $C$ has odd size and our colouring of this subgraph will have some monochromatic edges, and this will require a variation of Lemma \ref{unrootedlemma} that allows for a distinguished root vertex to receive special treatment.  We have not provided a separate proof of this lemma since it follows immediately from the same reductions used to prove Lemma \ref{unrootedlemma}.

\begin{lemma}
\label{rootedlemma}
Let $T$ be a cubic tree with $|V(T)| \ge 2$.  Let $(X,Y)$ be a split of $V_1(G)$, let $r \in X$, and let $\epsilon = \pm 1$.  There exists a split $(X^*,Y^*)$ of $T$ satisfying all of the following properties:
\begin{enumerate}
\item $X \subseteq X^*$ and $Y \subseteq Y^*$,
\item $\mathrm{disc}(X^*,Y^*) \in \epsilon \{-1,0,1,2,3\}$
\item The graph $H = T \setminus E(X^*,Y^*)$ satisfies one of the following:
\begin{enumerate}
\item  $\mathrm{deg}_H(r) = 1$ and $H$ has maximum degree 1.
\item $\mathrm{deg}_H(r) = 0$ and $H$ has at most one vertex of degree greater than 1.
\end{enumerate}
\end{enumerate}
\end{lemma}

With this lemma in place, we can now prove the second part of our main theorem (restated here).

\begin{theorem}
A cubic graph $G$ satisfies the Ban-Linial conjecture if it can be decomposed into a tree and a cycle.
\end{theorem}

\begin{proof}
We may assume $|V(G)| \ge 6$. Choose a tree $T$ and a cycle $C$ so that $\{ E(T), E(C) \}$ is a partition of $E(G)$. If $C$ has even-length, the theorem follows from Theorem \ref{treeBip}, and so we may assume $C$ has odd-length. 

Let $v,v'$ be cherries of $T$ and let $u$ be the common neighbour of $v,v'$ in $T$.  Let $H$ be the subgraph of $G$ consisting of the cycle $C$ together with the added vertex $u$ and the edges $uv, uv'$.  Following our convention, we choose a split $(X_0,Y_0)$ of $H$ that we treat as a colouring (with $X_0$ vertices black and $Y_0$ vertices white) in such a way that both $v,v'$ are black (and so $u$ is white), and there is a unique monochromatic edge, coloured black, one end of which is $v$.  Let $r$ be the other end of this monochromatic edge. Then $r \neq u$, but it could be that $r = v'$.

Let $T' = T - \{v,v'\}$. If $r = v'$, apply Lemma \ref{unrootedlemma} to $T'$ with split $(X_0 \cap V_1(T'), Y_0 \cap V_1(T'))$ and $\epsilon = -1$. If $r \neq v'$, apply Lemma \ref{rootedlemma} to the tree $T'$ with the split $(X_0 \cap V_1(T'), Y_0 \cap V_1(T'))$, the vertex $r$, and $\epsilon = -1$. 
In either case, let $(X^*,Y^*)$ be the resulting split of $T'$ and define $(X,Y) = (X_0 \cup X^*, Y_0 \cup Y^*)$.  Now $E(G) = E(T') \sqcup E(H)$, thus $\mathrm{disc}_G(X,Y) = \mathrm{disc}_{T'}(X^*,Y^*) + 1 \in \{ -2, \ldots , 2\}$ so Observation \ref{disc2bisect} implies that $(X,Y)$ is a bisection of $G$. It follows from Part 3 of Lemma \ref{unrootedlemma} or \ref{rootedlemma} as appropriate (and the split of $H$) that this bisection is nearly external, which means the result follows from Lemma \ref{nearlyexternal2banlinial}.
\end{proof}

\bibliographystyle{acm}
\bibliography{bib}

\end{document}